\documentclass[12pt,a4paper]{article}
\usepackage{}
\usepackage{verbatim}
\usepackage{setspace}
\usepackage[left=2.5cm,right=2.5cm,top=3.0cm,bottom=3.0cm,includeheadfoot]{geometry}
\usepackage{latexsym}
\usepackage{amssymb}
\usepackage{amsthm}
\usepackage{amsmath}
\usepackage{tikz}
\usetikzlibrary{matrix}
\usetikzlibrary{positioning}
\usetikzlibrary{arrows}
\usepackage{graphicx}
\usepackage{blindtext}

\usepackage[english]{babel}

\newtheorem{theorem}{Theorem}[section]
\newtheorem{lemma}[theorem]{Lemma}

\theoremstyle{definition}
\newtheorem{definition}[theorem]{Definition}
\newtheorem{example}[theorem]{Example}

\theoremstyle{proposition}
\newtheorem{proposition}[theorem]{Proposition}

\theoremstyle{remark}
\newtheorem{remark}[theorem]{Remark}
\theoremstyle{corollary}
\newtheorem{corollary}[theorem]{Corollary}
\theoremstyle{question}
\newtheorem{question}[theorem]{Question}

\newcommand{\field}[1]{\mathbb{#1}}
\newcommand{\C}{\field{C}}
\newcommand{\R}{\field{R}}
\newcommand{\N}{\field{N}}
\newcommand{\Z}{\field{Z}}
\newcommand{\Q}{\field{Q}}

\newcommand{\ham}{(M,\omega,\psi)}
\newcommand{\tor}{\mathcal{S}}

\usepackage[latin1]{inputenc}

\title{\textsc{Hamiltonian \(S^1\)-spaces with large equivariant pseudo-index}}
\author{Isabelle Charton}

\date{March 2019}
\begin{document}

\maketitle
	
	\begin{abstract}
	Let \((M,\omega)\) be a compact symplectic manifold of dimension \(2n\) endowed with a Hamiltonian circle action with only isolated fixed points. Whenever \(M\) admits a toric \(1\)-skeleton \(\mathcal{S}\), which is a special collection of embedded \(2\)-spheres in \(M\), we define the notion of equivariant pseudo-index of \(\mathcal{S}\): this is the minimum of the evaluation of the first Chern class \(c_1\) on the spheres of \(\mathcal{S}\). This can be seen as the analog in this category of the notion of pseudo-index for complex Fano varieties.\\
	In this paper we provide upper bounds for the equivariant pseudo-index. In particular, when the even Betti numbers of \(M\) are unimodal, we prove that it is at most \(n+1\) . Moreover, when it is exactly \(n+1\), \(M\) must be homotopically equivalent to \(\C P^n\).
	\end{abstract}

\footnotetext{
    2010 Mathematics Subjects Classification. 57R91, 57S25, 37J10\\
    Keywords and phrase. Hamiltonian circle actions, Fixed points, Equivariant Cohomology}

\tableofcontents

\begin{section}{Introduction}
	
	In the field of algebraic geometry  there are  various known characterizations of the projective space. Let \(X\) be a smooth complex projective variety of complex dimension \(n\). The variety  \(X\) is called Fano, if its anticanonical line bundle \(-\mathcal{K}_X\) is ample. The  pseudo-index of \(X\) is defined as
	\begin{align*}
	\iota_X= \min \left\lbrace c_1 (-\cal K_{\text{X}})[C]  \mid C \subset \text{X} \text{ is a rational curve}\right\rbrace ,
	\end{align*}
	where \(c_1 (-\cal K_{\text{X}})\) is the first Chern class of \(-\cal K_{\text{X}}\).
	It is known that the pseudo-index \(\iota_{X}\) is less or equal to \(n+1\) (see \cite{Mori}). It is natural to ask whether one can characterize a smooth Fano variety with maximal pseudo-index. The following was a  conjecture by Mori  and it was proven  by Cho, Miyaoka and Shepherd-Barron.
\begin{theorem}(\cite[Corollary 0.3]{pseudo})
		Let \(X\) be a smooth complex  projective Fano variety of complex dimension \(n\). If the pseudo-index \(	\iota_X  \) is equal to \(n+1\), then  \(X\) is isomorphic to \(\C P^n\).
	\end{theorem}
The goal of this article is to investigate  analogous questions for symplectic manifolds. More precisely, let \((M,\omega)\) be a symplectic manifold. It is known that \(M\) admits an almost complex structure \(J\) which is compatible with \(\omega\), i.e. \(\omega(\cdot ,J\cdot)\) is a Riemannian metric. Moreover, the space of such  structures is contractible. Hence, we can define complex invariants of the  tangent bundle \(TM\), for instance Chern classes. Let in particular \(c_1\) be the first Chern class of \(TM\). In analogy with the pseudo-index of a complex variety  we define the pseudo-index of \((M,\omega)\) to be
\begin{align*}
\rho_M = \inf \left\lbrace c_1[S] \mid S \text{\,\,is a  symplecticly embedded  \(2\)-sphere in \,}M  \right\rbrace.
\end{align*}

	\begin{question}\label{Q}
		Let \((M, \omega)\) be a compact connected symplectic manifold of (real) dimension \(2n\).\\
		\emph{(a)} Is \(n+1\) an upper bound for \(\rho_M\)?\\
		\emph{(b)} Does \(\rho_M=n+1\) imply that \(M\) is homotopy equivalent/ diffeomorphic / symplectomorphic \footnote{This means that \((M, \omega)\) is symplectomorphic to \((\C P^n, r\omega_{FS})\), where \(\omega_{FS}\) is the Fubini-Study symplectic form on \(\C P^n\) and \(r\) is a non-zero constant.} to  \(\C P^n\)?
	\end{question}
	
	In this paper we investigate the previous questions when the symplectic manifold is endowed with a  Hamiltonian \(S^1\)-action.
	
	So let \((M,\omega)\) be a symplectic manifold of real dimension \(2n\)  that can be endowed with a symplectic \(S^1\)-action. The \(S^1\)-action is called \textbf{Hamiltonian}  if there exists a smooth map \(\psi\colon M\rightarrow\R\), such that \(d\psi =-\iota_{\xi}\omega\), where \(\xi\) is the vector field on \(M\) generated by the \(S^1\)-action. The map \(\psi\) is called \textbf{moment map} and its set of critical points coincides with the \textbf{set of fixed points \(M^{S^1}\)} of the action.
	\begin{definition}
	Let \((M,\omega)\)  and \(\psi\colon M\rightarrow\R\) be as above. We call the triple \(\ham\) a \textbf{Hamiltonian \(S^1\)-space}, if the set of fixed points \(M^{S^1}\) is finite and the manifold \(M\) is compact and connected. 
	\end{definition}

In this category we introduce the notion of  \textbf{equivariant pseudo-index}. To define this we need the existence of a so called \textbf{toric \(1\)-skeleton}, as defined by Godinho, Sabatini and von Heymann in \cite{1224}. We recall that  a toric \(1\)-skeleton of a Hamiltonian \(S^1\)-space \(\ham\) of dimension \(2n\) is a collection \(\mathcal{S} \) of \(S^1\)-invariant symplecticly  embedded  \(2\)-spheres in \(M\), such that the Poincar\'e dual to their class in homology is the Chern class \(c_{n-1}\) (see Definition \ref{Def:toric1sk}). In \cite{1224} the authors give various conditions for the existence of a toric \(1\)-skeleton and remark that there are no known examples  of Hamiltonian \(S^1\)-spaces that do not admit a toric 1-skeleton. Whenever \(\ham\) admits a toric \(1\)-skeleton \(\mathcal{S}\), we define its \textbf{equivariant pseudo-index} by
\begin{align*}
\rho_M^{S^1}(\mathcal{S})= \min\left\lbrace c_1[S] \mid S\in \mathcal{S}\right\rbrace .
\end{align*}

	We focus on the case  when the vector of even Betti numbers of \(M\) is \textbf{unimodal} i.e. \(b_0\leq b_2 \leq \dots \leq b_{2\lfloor n / 2 \rfloor}\). Note that this  assumption  is not very restrictive. For example, the vector of even Betti numbers of a Hamiltonian \(S^1\)- space is unimodal, if the corresponding moment map  is
	index increasing \cite{Cho}.

	We prove the following upper bound for the equivariant pseudo-index.
	\begin{proposition}\label{Pro:boundpi}
		Let \(\ham\) be a  Hamiltonian \(S^1\)-space of dimension \(2n\), such that the vector of even Betti numbers of \(M\) is unimodal. If \(\ham\) admits a toric \(1\)-skeleton \(\mathcal{S}\), then the equivariant pseudo-index of \(\rho_M^{S^1}(\mathcal{S})\) is less or equal to \(n+1\).
	\end{proposition}

Since the equivariant pseudo-index of a toric \(1\)-skeleton is an upper bound for the pseudo-index, Proposition \ref{Pro:boundpi} answers Question \(\ref{Q}\) (a) for Hamiltonian \(S^1\)-spaces when there exists a toric \(1\)-skeleton and the vector of even Betti numbers is unimodal.

Moreover, let \(\ham\) be a Hamiltonian \(S^1\)-space and \(J\) an almost complex structure compatible with \(\omega\) which is also \(S^1\)-invariant. Hence for a fixed point \(P\in M^{S^1}\), we have a linear representation of \(S^1\) on \((T_PM,J_P)\simeq \C^n\). There exist complex coordinates \((z_1,\dots, z_n)\) for \((T_PM, J_P)\), such that the \(S^1\)-representation is given by 
\begin{align*}
\lambda \cdot (z_1,\dots, z_n)=(\lambda^{a_1} z_1, \dots ,\lambda^{a_n} z_n) \quad \text{for all } \lambda \in S^1,
 \end{align*}
 where  \(a_1, \dots, a_n \) are non-zero integers. These integers are called the \textbf{weights} of the \(S^1\)-action at \(P\).  The fixed point  \(P\) is a non-degenerate critical point of the moment map \(\psi\), whose (Morse-)index is equal to twice the number of negative weights at \(P\). Hence, \(\psi: M \rightarrow \R\) is a Morse-function with critical points of just even index. So the  odd Betti numbers of \(M\) are all equal to zero and the \(2i\)-th Betti number is equal to the number fixed points with precisely \(i\) negative weights. This implies that \footnote{Let \((M,\omega)\) be a symplectic compact connected manifold of dimension \(2n\). Then \(\omega^n\) is a volume form on \(M\). Hence \(\omega^i\) defines a non-zero element in the \(2i\)-th de Rham group of \(M\) and the \(2i\)-th Betti number of \(M\) is greater or equal than 1 for all \(i\leq n\). }

  \emph{ A Hamiltonian \(S^1\)-space \(\ham \) of dimension \(2n\) has at least \(n+1\) fixed points and if the number of fixed points is minimal then \(M\)  has the same Betti numbers as \(\C P^n\).}

 \begin{example}  A standard \(S^1\)-action on \(\C P^n\)  is given by
 	\begin{align*}
 	\lambda \cdot [z_0,\dots , z_n]=[\lambda^{m_0}z_0, \dots ,\lambda^{m_n}z_n] \text{ for all } \lambda\in S^1,
 	\end{align*}
 	where \(m_0,\dots, m_n \in \Z\). The action has only isolated fixed points if and only if these integers are pairwise different. In this case the  fixed points are 
 	\begin{align*}
 	P_j=[0,\dots,0,1,0\dots,0] \quad (1 \text{ at the j-th factor and } 0\leq j\leq n)
 	\end{align*}
 	and the set of weights a the fixed point \(P_j\) is \(\left\lbrace m_i-m_j \right\rbrace_{i\neq j}\).
 \end{example}

	The main results of this paper are stated in the following theorems.
	\begin{theorem}\label{Thm:Main1}
		Let \(\ham\) be a  Hamiltonian \(S^1\)-space of dimension \(2n\) which admits a toric \(1\)-skeleton \(\mathcal{S} \) with equivariant pseudo-index \(\rho_M^{S^1}(\mathcal{S})=n+1\). Assume that \(n\leq 5\) or the even Betti numbers of \(M\) are unimodal. Then \(M\) has the same Betti numbers as \(\C P^n\) and the \(S^1\)-representations at the fixed  points are the same as those of a standard \(S^1\)-action on \(\C P^n\).\\
	\end{theorem}
	
	We prove that Theorem \ref{Thm:Main1} has the following stronger consequence, which  answers  Question \ref{Q} \((b)\).
	
	\begin{theorem}\label{Thm:Main2}
	Under the hypotheses of Theorem \ref{Thm:Main1}, \(M\) is homotopy equivalent to \(\C P^n\).
	\end{theorem}

In low dimension the results are stronger. Namely, by Karshon's classification of Hamiltonian \(S^1\)-spaces of dimension \(4\) (\cite{12}) and by a theorem of Wall for six-dimensional manifolds (\cite{wall}), we obtain the following corollary.

\begin{corollary}
Let \(\ham\) be a  Hamiltonian \(S^1\)-space of dimension \(2n\) which admits a toric \(1\)-skeleton \(\mathcal{S} \) with equivariant pseudo-index \(\rho_M^{S^1}(\mathcal{S})=n+1\). \\
\emph{(a)} If \(n=2\), then \((M,\omega)\) is symplectomorphic to \((\C P^2, r\omega_{FS})\), where \(\omega_{FS}\) is the Fubini-Study symplectic form on \(\C P^2\) and \(r\) is a positive constant.\\
\emph{(b)} If \(n=3\), then \(M\) is diffeomorphic to \(\C P^3\).

\end{corollary}

Here is a brief overview of the structure of this paper. In Section 2 we review some important results needed in this work and explain the relation between the Betti numbers and the equivariant pseudo-index. In Section 3  we discuss upper bounds for the equivariant  pseudo-index. In the last sections we prove Theorem \ref{Thm:Main1} and \ref{Thm:Main2} .\\

\textbf{Acknowledgements.} This work is part of the SFB/TRR 191 'Symplectic Structures in Geometry, Algebra and Dynamics', funded by the DFG. I would like to thank Silvia Sabatini for very helpful discussions and comments.
\end{section}

\begin{section}{Background material}
	
	In this section we briefly review some background material needed in the following sections.

\begin{subsection}{Equivariant Cohomology}
       (We refer to \cite{AB} and \cite{supersymme} for an extensive introduction to equivariant cohomology.)
       
       Let \(S^{\infty}\)  be the unit sphere in \(\C^{\infty}\). Up to homotopy equivalence \(S^{\infty}\) is the only contractible space on which \(S^1\) acts freely. Now let \(M\) be a manifold endowed with a smooth \(S^1\)-action. In the Borel-model the \(S^1\)-equivariant cohomology of  \(M\) is defined as follows. The diagonal action of \(S^1\) on \(M\times S^{\infty}\) is free. By   \(M \times_{S^1} S^{\infty}\) we denote the orbit space. The \(S^1\)-equivariant cohomology ring of \(M\) is
       \begin{align*}
       H_{S^1}^*(M;R)\colon =H^*(M \times_{S^1} S^{\infty};R),
       \end{align*}
       where \(R\) is the coefficient ring.\\
       Let \(M^{S^1}\) be the set of fixed points and assume it is not empty. Let \(F\) be one of its connected components. The inclusion map \(i_F\colon F \rightarrow M\) is an \(S^1\)-equivariant  map, so it induces a map
       \begin{align*}
       i_F^*\colon H_{S^1}^*(M) \rightarrow H_{S^1}^*(F).
       \end{align*}

        Moreover, the projection \(\ M \times_{S^1} S^\infty\ \rightarrow \C P^\infty \) induces a push-forward map in equivariant cohomology
        \begin{align*}
         H_{S^1}^*(M)\rightarrow H^{*-\operatorname{dim}(M)}(\C P^\infty),
        \end{align*}
        which can be seen as integration along the fibers. So we denote it by \(\int_M\). The following theorem, due to Atiyah-Bott and Berline-Vergne (see \cite{AB}, \cite{BV}) gives a formula for the map \(\int_M\) in terms of fixed point set data.

        \begin{theorem}\label{Thm:ABBV}(ABBV Localization formula) Let \(M\) be a compact oriented manifold endowed with a smooth \(S^1\)-action. Given \(\mu\in H_{S^1}^* (M;\Q) \)
                \begin{align*}
                \int_M \mu \,=\, \sum_{F\subset M^{S^1}} \int_F \dfrac{i_F^*(\mu)}{\text{e}^{S^1}(N_F)},
                \end{align*}
                where the sum runs over all connected components \(F\) of \(M^{S^1}\) and \(\text{e}^{S^1}(N_F)\) is the equivariant Euler class of the normal bundle
                 \(N_F\) to \(F\).
        \end{theorem}

        In particular, if \(P_i\in M^{S^1}\) is an isolated fixed point, then \(N_{\left\lbrace P_i\right\rbrace }=T_{P_i}M\) and
        \begin{align*}
        e^{S^1}(T_{P_i}M)=\left( \prod_{j=1}^{n}\omega_{ij}\right) x^n \in H^*_{S^1}(\left\lbrace P_i\right\rbrace , \Z)=\Z[x],
        \end{align*}
        where \(\omega_{i1},\dots,\omega_{in}\) are the weights\footnote{Note that the signs of the individual weights are not well-defined, but the sign of their  product is.} of the \(S^1\)-representation on \(T_{P_i}M\).
        Hence, if the set of fixed points is isolated we have the following corollary.

        \begin{corollary}\label{Cor:ABBVisol}
                Let \(M\)  be a compact oriented manifold endowed with a smooth \(S^1\)-action such that \(M^{S^1}=\left\lbrace P_0,\dots P_N \right\rbrace \). Given \(\alpha \in H_{S^1}^*(M;\Q)\)
                \begin{align*}
                \int_M \alpha = \sum_{i=0}^{N}\dfrac{\alpha(P_i)}{\left( \prod_{j=1}^{n}\omega_{ij}\right) x^n},
                \end{align*}
                where \(\omega_{i1},\dots,\omega_{in}\) are the weights of the \(S^1\)-representation on \(T_{P_i}M\) and \(\alpha(P_i)=i^*_{P_i}(\alpha)\).
        \end{corollary}

Now let \((M,\omega)\) be a compact symplectic manifold of dimension \(2n\) endowed with a Hamiltonian \(S^1\)-action. In \cite{Kirwan}, Kirwan proves important properties of the equivariant cohomology ring of this space. 
For the case that  \(M\) has the same cohomology groups as \(\C P^n\),  Tolman \cite{TolmanPetri} uses these properties and   gives an explicit basis for the equivariant and ordinary cohomology rings  in terms of the fixed point data and the (equivariant) first Chern class. The following lemma contains results of  \cite[Section 3]{TolmanPetri}.

\begin{lemma}\label{Lem:Tolman}
Let \(\ham \) be a Hamiltonian \(S^1\)-space  of dimension \(2n\) with exactly \(n+1\) fixed points. For \(i=0,\dots,n \), let \(P_i\) be the unique fixed point with exactly \(i\) negative weights. We denote by \(\Gamma_i\) the sum of all weights at \(P_i\) and by \(\Lambda_i^-\) the product\footnote{The empty product \(\Lambda_0^-\) is equal to 1.} of all negative weights at \(P_i\).  Then the following holds: A  generator of \(H^{2i}(M;\Z)=\Z\) is given by
	\begin{align*}
	\tau_i = \prod_{j=0}^{i-1}\frac{\Lambda_i^-}{\Gamma_i-\Gamma_j} c_1^i.
	\end{align*}
	Moreover, for \(i,j\in \left\lbrace 0,\dots n\right\rbrace \) we have  \(\Gamma_i>\Gamma_j\) if and only if \(i<j\).
\end{lemma}

\begin{remark}\label{Rem:ring}
Let \(\ham\) be a Hamiltonian \(S^1\)-space of dimension \(2n\) with \(n+1\) fixed points. Then \(M\) has the same (co-)homology groups as \(\C P^n\). Moreover,  the ring structure of \(H^*(M;\Z)\) can be easily recovered from the fixed point set data by using the results of  Lemma \ref{Lem:Tolman}.
\end{remark} 

The discussion in this remark, together with the results in Theorem \ref{Thm:Main1}, have the following straightforward consequence.
\begin{corollary}\label{Cor:Ring}
Let \(\ham\) be a Hamiltonian \(S^1\)-space of dimension \(2n\) which satisfies the hypotheses of Theorem \ref{Thm:Main1}. Then \(M\) has the same integer cohomology ring as \(\C P^n\).
\end{corollary}

This corollary is the key ingredient of the proof of Theorem \ref{Thm:Main2}.

\end{subsection}

        \begin{subsection}{Toric One-Skeletons}\label{Sec:toric}
In this section we review some material, which we adapt from \cite{new} and \cite{1224}.
                Let \(\ham\) be a Hamiltonian \(S^1\)-space with fixed point set \(M^{S^1}=\left\lbrace P_0,\dots, P_N\right\rbrace \). We denote the weights at the fixed point \(P_i\) by \(\omega_{i1},\dots, \omega_{in}\) (repeated with multiplicity). The multiset of positive weights \(W_+\) (resp. negative weights \(W_-\)) associated to \(\ham\) is the multiset\footnote{The symbol \(\biguplus \) denotes the union of multisets. }
                \begin{align*}
                \biguplus_{P_i \in M^{S^1}}\left\lbrace \omega_{ik}\,;\, \omega_{ik}>0\right\rbrace \quad ( \text{resp.}\biguplus_{P_i \in M^{S^1}}\left\lbrace \omega_{ik}\,;\, \omega_{ik}<0\right\rbrace ).
                \end{align*}
                
                The next lemma is the key ingredient behind the definition of a toric \(1-\)skeleton.
                
                \begin{lemma}\label{Lem:SymmetryWeights}
                        Let \(\ham\) be a  Hamiltonian \(S^1\)-space  and \(W_+\) and \(W_-\) be the multisets of positive and negative  weights. If an integer \(k>0\) belongs to \(W_+\) with multiplicity  \(m\) then \(-k\) belongs to \(W_-\) with the same multiplicity, i.e. \(W_+=-W_-\).
                \end{lemma}
                
                This lemma was  proved  by Hattori \cite[Proposition 2.11]{hattori} for almost complex manifolds. In particular, there exists a bijection \(f\colon W_+\rightarrow W_-\) , such that
                \begin{align*}
                f(\omega_{ik})=\omega_{jl} \quad \text{implies} \quad \omega_{ik}=-\omega_{jl}.
                \end{align*}
 \begin{definition}
Let \(\ham\) be a Hamiltonian \(S^1\)-space. An  oriented graph \(\Gamma=(V,E)\) is associated to \(\ham\), if there exists a bijection \(f:W_+\rightarrow W_-\) as above, such that
                        \begin{align*}
                        &\cdot \text{The vertex set }V \text{ is the fixed point set }M^{S^1}.\\
                        &\cdot \text{The edge set is } E=\left\lbrace e_{ik} \, ;\, \omega_{ik} \in W_+ \right\rbrace ,\text { such that } e_{ik} \text{ is the oriented edge from } P_i \text{ to } P_j,\\
                        &\quad  \text{ where } f(\omega_{ik})=\omega_{jl}.\\
                        &\cdot \text{We label this graph by  }\,\operatorname{w}:E\rightarrow \Z_{>0}, \text{ where} \operatorname{w} \text{ sends an edge }e=e_{ik} \text{ to the weight }\\
                        &\quad \operatorname{w}(e)=\omega_{ik}. \text{ We call } \operatorname{w} \text{ the weight map.}
                        \end{align*}
                    \end{definition}

\begin{definition}\label{Def:toric1sk}
Let \(\ham\)  be a Hamiltonian \(S^1\)-space. We say  \(\ham\) \textbf{admits a toric \(1\)-skeleton} if there exists an oriented graph \(\Gamma=(V,E)\) associated to \(\ham\) satisfying the following property:
                        \begin{quote}
                        For each oriented edge \(e\in E\) from \(P_i\) to \(P_j\), labeled by \(\operatorname{w}(e)\in \Z_{>0}\), there exists a smoothly embedded, symplectic, \(S^1\)-invariant \(2\)-sphere fixed by \(\Z_{\operatorname{w}(e)}\). Moreover, \(S^1\) acts on this \(2\)-sphere with fixed points \(P_i\) and \(P_j\) and \(\operatorname{w}(e)\) resp. \(-\operatorname{w}(e)\) is the weight at \(S^1\)-representation at \(P_i\) resp. \(P_j\).
                        \end{quote}
                     Let \(\Gamma=(V,E)\) be such a graph and for each \(e\in E\) let \(S_e\) be a 2-sphere satisfying the properties above. Then \(\mathcal{S}=\left\lbrace S_e \mid e\in E\right\rbrace \) is \textbf{a toric \(1\)-skeleton of \(\ham\)} associated to \(\Gamma\).

               Moreover, the \textbf{equivariant pseudo-index}  of \(\tor\) is defined as
                       \begin{align*}
                       \rho_M^{S^1}(\mathcal{S}) =\min \left\lbrace c_1[S] \mid S \in \tor \right\rbrace .
                       \end{align*}
                    \end{definition}

                The importance of introducing the concept of toric \(1\)-skeletons relies in the following. The class in homology  of the toric \(1\)-skeleton is  the Poincar\'{e} dual of the Chern class \(c_{n-1}(M)\in H^{2n-2}(M;\Z)\) \cite[Lemma 4.13]{1224}. Moreover, the Chern number \(\int_M c_1c_{n-1}\) depends only on the Betti numbers of \(M\). 
                
                 \begin{proposition}\label{Prop:ChernBetti}\cite[Corollary 3.1]{new}
                	Let \(\ham \) be a Hamiltonian \(S^1\)-space of dimension \(2n\). Then the Chern number \(\int_M c_1c_{n-1}\) depends only the Betti numbers of \(M\). In particular,
                	\begin{align*}
                	\int_{M}c_1c_{n-1}=\sum_{k=0}^{n}\left[ 6k(k-1)+\dfrac{5n-3n^2}{2} \right] b_{2k}.
                	\end{align*}
                \end{proposition}
           The  following corollary of Proposition \ref{Prop:ChernBetti} gives us relations between the Betti numbers of \(M\) and the equivariant pseudo-index.

                \begin{corollary} \label{Cor:Cinteger}(c.f. \cite[Corollary 5.5]{1224})
                        Let  \((M,\omega ,\psi) \) be a Hamiltonian \(S^1\)-space of dimension \(2n\) and let \(\textbf{b}=(b_0,b_2,\dots ,b_{2n})\) be the vector of even Betti numbers of \(M\). Assume that \(\ham\) admits a toric \(1\)-skeleton \(\mathcal{S} \) with equivariant pseudo-index \(\rho:=\rho_M^{S^1}(\mathcal{S})\). Consider the integer \(\mathcal{C}(\rho,n, \textbf{b})\), defined as,
                        \[
                        \mathcal{C}(\rho,n, \textbf{b})=\begin{cases}

                        \sum_{k=1}^{\frac{n}{2}}\left[ 12k^2-n(\rho+1)\right]b_{n-2k}
                        -\dfrac{n}{2}(\rho +1)b_n,

                        & \text{for \(n\) even,} \\
                        \sum_{k=1}^{\frac{n-1}{2}}\left[ 12k(k+1)+3-n(\rho+1)\right]b_{n-1-2k}
                        -\left[ n(\rho+1)-3\right] b_{n-1},& \text{for  n is odd. }
                        \end{cases}
                        \]
                        It is  non-negative and vanishes if and only if \(c_1[S]=\rho\)  for all \(S \in \mathcal{S}\).
                \end{corollary}

            \begin{proof}
            The cardinality of \(\mathcal{S}\)  is equal to \(\frac{n}{2} \left|  M^{S^1} \right|   = \frac{n}{2} \chi(M)\), where \(\chi(M)\) is the Euler characteristic of \(M\). Thus
            	\begin{align}\label{Eq:C} 
            	\sum_{S\in \mathcal{S}}c_1[S]-\rho\dfrac{n}{2} \chi (M)
            	\end{align}
            	is  non-negative  and it is zero if and only if \(c_1[S]=\rho \) for all \(S \in \mathcal{S}\).
            	For \(k\) odd \(b_k\) is equal to \(0\) and the Poincar\'e duality implies \(b_{2k}=b_{2(n-k)}\),  so
            	\begin{align*}
            	\chi(M)= \begin{cases}
            	2\sum_{k=1}^{\frac{n}{2}}b_{n-2k} +b_n,
            	& \text{for \(n\) even}, \\
            	2\sum_{k=1}^{\frac{n-1}{2}}b_{n-1-2k}, & \text{for  n is odd. }
            	\end{cases}
            \end{align*}
            Moreover,
            \begin{align*}
            \sum_{S\in \mathcal{S}}c_1[S]= \int_ M c_1c_{n-1},
            \end{align*}
            since \(\mathcal{S}\) is Poincar\'e dual to the Chern class \(c_{n-1}\).
            	By using Proposition \ref{Prop:ChernBetti} we conclude that \eqref{Eq:C} is equal to \(\mathcal{C}(\rho,n,\textbf{b})\) and the claim follows.
            \end{proof}

            \end{subsection}
            \end{section}

           \begin{section}{Upper Bounds for the Equivariant Pseudo-Index}
           	In this section we prove upper bounds for the equivariant pseudo-index of a toric \(1\)-skeleton.

           	\begin{lemma}\label{Lem:pi}
           		Let \(\ham\)  be a Hamiltonian \(S^1\)-space of dimension \(2n\) which admits a toric \(1\)-skeleton \(\mathcal{S}\). Then the equivariant pseudo-index \(\rho_M^{S^1}(\mathcal{S})\) of \(\mathcal{S} \) satisfies \(\rho_M^{S^1}(\mathcal{S})\leq 2n\).
           	\end{lemma}
           	
           	\begin{proof}
           		Let \(M^{S^1}=\left\lbrace P_0,\dots,P_N\right\rbrace \) be the set of fixed points, where \(P_0\) resp. \(P_N\) is the minimum resp. the maximum of the moment map \(\psi\).
           		
           		For any \(i=0,\dots,N\) let \(\Gamma_i\) be the sum of the weights at the fixed point \(P_i\) and let \(\lambda_i\) be the number of negative weights at the fixed point \(P_i\). Note that \(\lambda_0=0\),  \(\lambda_N=n\) and \(0<\lambda_i<n\) for \(i\neq0,N\). \\
           		Now let \(W_+\) resp. \(W_-\) be the multiset of  all positive  resp. negative weights of the \(S^1\)-action. We set \(\alpha=\text{max}\left\lbrace \omega_{ij}\mid \omega_{ij}\in W_+\right\rbrace \). By Lemma \ref{Lem:SymmetryWeights} we have \(-\alpha =\text{min}\left\lbrace \omega_{ij}\mid \omega_{ij}\in W_-\right\rbrace \). In particular, \(1\leq \omega_{ij}\leq\alpha \) if \(\omega_{ij}\in W_+\) and \(-\alpha\leq \omega_{ij}\leq -1\) if \(\omega_{ij}\in W_-\). We conclude
           		\begin{align*}
           		(n-\lambda_i)-\alpha \lambda_i \leq \Gamma_i \leq \alpha(n-\lambda_i) -\lambda_i,
           		\end{align*}
           		for all \(i=0,\dots,N\).\\
           		Moreover, there exists a fixed point \(P_k\in M^{S^1}\) such that \(\alpha\) is a weight at \(P_k\). Let \(S \) be the corresponding \(2\)-sphere in the toric \(1\)-skeleton \(\mathcal{S}\), i.e. \(S^1\) acts on \(S\) with fixed points \(P_k\) and \(P_j\) and \(\alpha\) resp. \(-\alpha\) is the weight of the \(S^1\)-representation on \(T_{P_k}S\) resp. \(T_{P_j}S\). The ABBV formula gives us
           		\begin{align*}
           		c_1\left[ S\right] &= \dfrac{c_1^{S^1}(P_k)-c_1^{S^1}(P_j)}{\alpha}=\dfrac{\Gamma_k-\Gamma_j}{\alpha}\\
           		&\leq \dfrac{\alpha(n-\lambda_k) -\lambda_k+\alpha\lambda_j-(n-\lambda_j)}{\alpha}\\
           		&= n+ \lambda_j - \lambda_k + \frac{\lambda_j-\lambda_k-n}{\alpha}\leq 2n,
           		\end{align*}
           		and the claim  follows.
           	     \end{proof}

           	 Recall the definition of the index of an almost complex manifold.
                   \begin{definition}
                   Let \((M,J)\) be a compact connected  almost complex manifold and \(c_1\) be its first Chern class. The index \(k_0\) of \((M,J)\)  is the largest integer, such that
                   \begin{align*}
                   c_1=k_0 \eta \quad \text{modulo torsion}
                   \end{align*}
                   for some non-torsion element \(\eta \) in \(H^2(M;\Z)\).
                   \end{definition}

                   In \cite{sabatinichern}, Sabatini proves the following result.

                   \begin{proposition}\label{Pro:Sabatini}
            Let \(\ham\) be a Hamiltonian \(S^1\)-space of dimension \(2n\), then the index \(k_0\) satisfies the following inequalities
            \begin{align*}
            1\leq k_0 \leq n+1.
            \end{align*}
                   \end{proposition}

               Obviously, the equivariant pseudo-index of a toric \(1\)-skeleton must be a multiple of the index. Thus, from   Lemma \ref{Lem:pi} and Proposition \ref{Pro:Sabatini} we obtain the following corollary. 
               
               \begin{corollary} Let \(\ham \) be a Hamiltonian \(S^1\)-space of dimension \(2n\)  which admits a toric \(1\)-skeleton \(\mathcal{S}\) with equivariant pseudo-index \(\rho_M^{S^1}(\mathcal{S})>0\). Moreover, let  \(k_0\) be the index of \((M,\omega)\). Under these assumptions \(k_0=n+1\) implies  \(k_0=\rho_M^{S^1}(\mathcal{S})\).
               \end{corollary}
           
           Moreover, if we assume that the vector of even Betti numbers of \(M\) is unimodal, we obtain a stronger upper bound for the equivariant  pseudo-index (see Proposition \ref{Pro:boundpi}). The proof of this proposition  is very similar to that of \cite[Corollary 5.8]{1224} and it is recalled here for the sake of completeness.

                   \begin{proof}[Proof of Proposition \ref{Pro:boundpi}] 
                   	  Let \(\ham\) be a Hamiltonian \(S^1\)-space of dimension \(2n\) which admits a toric \(1\)-skeleton \(\mathcal{S}\) with equivariant pseudo-index \(\rho:=\rho_M^{S^1}(\mathcal{S})\).
                   	We can see  \(\mathcal{C}(\rho,n,\textbf{b})\) (as in Corollary \ref{Cor:Cinteger}) as a linear function of the vector  \(\textbf{b}\) of the even Betti numbers of \(M\)
                   	\begin{align*}
                   	\mathcal{C}(\rho\,,n,\textbf{b})=\sum_{i=0}^{\lfloor n / 2 \rfloor}A_i(\rho,n)b_{2i}.
                   	\end{align*}
                   	Let
                   	\begin{align*}
                   	M(\rho, n)\colon= \sum_{i=0}^{\lfloor n / 2 \rfloor} A_i(\rho,n)= \frac{1}{2}n (n+1)(n+1-\rho)
                   	\end{align*}

                   	and note that
                   	\begin{align*}\label{Eq:A}
                   	A_{\lfloor n / 2 \rfloor-1}(\rho,n)<\dots<A_0(\rho,n)
                   	\end{align*}

                   	and \(A_{\lfloor n / 2 \rfloor}(\rho,n)\leq 0\) if \(\rho\geq 2\). Moreover, let
                   	\begin{align*}
                   	\lambda\colon = \lambda (\rho, n)=
                   	\min \left\lbrace i\in \left\lbrace 0,1,\dots, \lfloor n / 2 \rfloor\right\rbrace \mid A_i(\rho,n)\leq 0 \right\rbrace.
                   	\end{align*}
                   	Now assuming that \(\rho>n+1\) and that the vector of even Betti numbers is unimodal, i.e.  \(1=b_0\leq\dots\leq b_{2 \lfloor n/2 \rfloor}\), we have
                   	\begin{align*}
                   	\mathcal{C}(\rho\,,n,\textbf{b})&=\sum_{i=0}^{\lfloor n / 2 \rfloor}A_i(\rho,n)b_{2i}\\
                   	&= \sum_{i=0}^{\lambda-1}A_i(\rho,n)b_{2i} +\sum_{i=\lambda}^{\lfloor n / 2 \rfloor}A_i(\rho,n)b_{2i}\\
                   	&\leq b_{2\lambda} \sum_{i=0}^{\lfloor n / 2 \rfloor}A_i(\rho,n) = b_{2\lambda} M(\rho, n)\\
                   	&= b_{2\lambda} \frac{1}{2}n (n+1)(n+1-\rho)<0.
                   	\end{align*}
                   	
                   	Thus \(\mathcal{C}(\rho\,,n,\textbf{b})<0\) if the vector of even Betti numbers of \(M\) is unimodal and \(\rho>n+1\). This contradicts  Corollary \ref{Cor:Cinteger}.
                           \end{proof}

           \end{section}

            \begin{section}{The Weights: Proof of Theorem 1.6 }
        	In this section we prove  Theorem \ref{Thm:Main1}.
        	
        	The following corollary follows from Corollary \ref{Cor:Cinteger}. The proof is very similar to the one of  \cite[Corollary 5.7 and 5.8]{1224}.
                
             \begin{corollary}\label{Cor:Bettino}
             Let \(\ham\) be a Hamiltonian \(S^1\)-space of dimension \(2n\), which admits a toric \(1\)-skeleton \(\mathcal{S}\) with equivariant pseudo-index \(\rho_M^{S^1}(\mathcal{S})=n+1\). Assume that \(n\leq 5\) or that the vector of even Betti numbers of \(M\) is unimodal. Then \(M\) has the same Betti numbers as \(\C P^n\).
             \end{corollary}  
        In order to describe the \(S^1\)-representations on \(TM\mid_{M^{S^1}}\) of a Hamiltonian \(S^1\)-space that satisfies the assumptions of Theorem \ref{Thm:Main1} we use ideas from Hattori's work \cite[Section 3 and 4]{hattori}. However, the existence of a toric \(1\)-skeleton simplifies the proof  (see Remark \ref{Rem:Hattori}).

                \begin{corollary}\label{Cor:allesn+1}
                Let \(\ham\) be a Hamiltonian \(S^1\)-space of dimension \(2n\), which admits a toric \(1\)-skeleton \(\mathcal{S}\) with equivariant pseudo-index \(\rho_M^{S^1}(\mathcal{S})=n+1\). If \(M\) has the same Betti numbers as \(\C P^n\), then \(c_1[S]=n+1\) for all \(S \in \mathcal{S}\).
                \end{corollary}
                \begin{proof}
                	If \(M\) has the same Betti number as \(\C P^n\), then \(\mathcal{C}(\rho=n+1 ,n, \textbf{b})=0\). Hence, Corollary \ref{Cor:Cinteger} implies that \(c_1[S]=n+1\) for all \(S \in \mathcal{S}\).
               
               \end{proof}

                \begin{lemma}\label{Lem:sumGamma}
                	
                	Let \(\ham\) be a Hamiltonian \(S^1\)-space of dimension \(2n\). Suppose that \(\ham\) admits a toric \(1\)-skeleton \(\mathcal{S}\) with equivariant pseudo-index \(\rho_M^{S^1}(\mathcal{S})=n+1\) and that \(M\) has the same Betti numbers as \(\C P^n\). For \(i=0,\dots,n\) we denote by  \(P_i\)  the unique fixed point with exactly \(i\) negative weights and we denote by  \(\Gamma_i\)  the sum of all weights at \(P_i\).\\ 
                	Under these assumptions, there exist integers \(a_0>\dots>a_n\) and \(d\), such that
                        \begin{align*}
                        \Gamma_i=(n+1)a_i+d  \quad \text{for all }i=0,\dots, n.
                        \end{align*}
                        Moreover, we have
                        \begin{align*}
                        \Gamma_i =\sum_{j=0}^{n}(a_i-a_j)  \quad \text{for all }i=0,\dots, n.
                        \end{align*}
                        
                \end{lemma} 
                \begin{proof}(c.f.\cite[Lemma 3.19]{hattori})
                        First we show that \(n+1\) is a divisor of \(\Gamma_i-\Gamma_0\) for all \(i=1,2,\dots, n\).
                        Then we can choose \(a_0=0\) and \(d=\Gamma_0\). Hence, \(a_i\) is given by \(\frac{\Gamma_i-\Gamma_0}{n+1}\), and  \(a_0>a_1> \dots >a_n\) follows from \(\Gamma_0 >\Gamma_1 >\dots >\Gamma_n\) (see Lemma \ref{Lem:Tolman}).

                        Now we show that \(n+1\) is a divisor of \(\Gamma_i-\Gamma_0\) by induction.

                        Consider the fixed point  \(P_1\) with one negative weight.  Let  \(m\) be the negative weight at \(P_1\) and let \(S \in \mathcal{S}\) be the corresponding \(2\)-sphere in \(\mathcal{S}\), i.e. \(S^1\) acts on \(S\) with fixed points \(P_1\) and \(P_j\) and \(m\) resp. \(-m\) is the weight of the \(S^1\)-representation at \(P_1\) resp. \(P_j\), (where \(P_j\) is a fixed point of the \(S^1\)-action on \(M\)). Then the ABBV formula gives us
                        \[c_1[S]=\dfrac{c_1^{S^1}(P_1)-c_1^{S^1}(P_j)}{m}=\dfrac{\Gamma_1-\Gamma_j}{m} .\]
                        By Corollary \ref{Cor:allesn+1}, we have \(c_1[S]=n+1\).
                        So \(n+1\) is a divisor of \(\frac{1}{m}(\Gamma_1-\Gamma_j)\). Since \(m\) is negative, we have \(\Gamma_1-\Gamma_j<0\). So \(j\) must be \(0\) and \(n+1\) is a divisor of \(\Gamma_1-\Gamma_0\).\\
                        Now assume that \(n+1\) is a divisor of \(\Gamma_i-\Gamma_0\) for all \(i=1,2,\dots, j\) and \(j<n\). Consider the fixed point \(P_{j+1}\) with \(j+1\) negative weights. With the same argument  as above it follows that \(n+1\) is a divisor of \(\Gamma_{j+1}-\Gamma_i\) for some \(i=0,1,\dots,j\). Hence, \(n+1\) is a divisor of \(\Gamma_{j+1}-\Gamma_{0}\).
                        The first claim follows.\\
                        Now let us fix integers \(a_0,a_1,\dots , a_n\) and \(d\) such that \(\Gamma_i=(n+1)a_i+d\). By Lemma \ref{Lem:SymmetryWeights}, we have
                        \[0=\sum_{j=0}^{n}\Gamma_j=\sum_{j=0}^{n}\left((n+1)a_j+d \right)=(n+1)\sum_{j=0}^{n}a_j \quad +(n+1)\, d. \] It follows
                        \[d=-\sum_{j=0}^{n}a_j,\]
                        and
                        \[\Gamma_i= (n+1) \, a_i- \sum_{j=0}^{n}a_j=\sum_{j=0}^{n}(a_i-a_j).\]
                        \end{proof}
                    
                    Now we show   under the assumption of Lemma \(\ref{Lem:sumGamma}\), that the weights at the fixed point   \(P_i\) are given by \(\left\lbrace a_i-a_j\right\rbrace _{i\neq j}\).

                        The following lemma is an application of the Atiyah-Segal formula \cite{AtSe} in equivariant K-theory. A proof of this lemma in a slight different version can be found in  \cite[Proof of Lemma 3.6]{hattori}.
                        \begin{lemma}\label{Lem:Laurent}
                                Assume that  the hypotheses of Lemma \ref{Lem:sumGamma} hold. Denote by \(\omega_{i1},\dots \omega_{in}\) the weights at the fixed point \(P_i\). The         function
                                \begin{align*}
                                \varphi_i(t)=\frac{\prod_{j\neq i}\left( 1-t^{a_i-a_j}\right)  }
                                {\prod_{k=1}^{n}\left( 1-t^{\omega_{ik}}\right) }
                                \end{align*}
                                belongs to the Laurent polynomial ring \(\Z \left[ t, \, t^{-1} \right] \) for each \(i=0,\dots,n\).
                        \end{lemma}

                        \begin{remark}\label{Rem:Hattori}
                       Consider the situation of Lemma \ref{Lem:sumGamma}. By  Hattori's result \cite[Theorem 5.7]{hattori}, we could already conclude that the weights at the fixed points \(P_i\)  are given by \(\left\lbrace a_i-a_j\right\rbrace _{j\neq i}\). However, in the setting of Theorem \ref{Thm:Main1} we can simplify the proof of Hattori. The key point for this is the following easy corollary of Lemma \ref{Lem:sumGamma}.
                       \end{remark}

                \begin{corollary}\label{Co:possibleweights}
                        Assume that the hypotheses of Lemma \ref{Lem:sumGamma} hold. Let us fix integers
                        \(a_0>\dots >a_n,\) such that
                        \begin{align*}
                        \Gamma_i=\sum_{j=0}^{n}(a_i-a_j) \quad \text{for all } i=0,\dots,n.
                        \end{align*}
                        If \(m\) is a weight at the fixed point \(P_i\), then \(m=a_i-a_k\) for some \(k\neq i\)\ and \(-m=a_k-a_i\) is a weight at the fixed point \(P_k\).
                \end{corollary}
                \begin{proof}
                        Let \(P_i\) be a fixed point and \(m\) be a weight at \(P_i\). Consider the corresponding \(2\)-sphere \(S \in\mathcal{S}\). \(P_i\) is a fixed point of the \(S^1\)-action on \(S\). Let \(P_k\) be the second fixed point of the \(S^1\)-action on    \(S\), so \(-m\) is a weight at \(P_k\). By Corollary \ref{Cor:allesn+1}, we have \(c_1[S]=n+1\). From  the ABBV formula it follows
                        \[c_1[S]=\frac{1}{m}(\Gamma_i-\Gamma_k),\] but
                        \[\Gamma_i-\Gamma_k=\sum_{j=0}^{n}(a_i-a_j)-\sum_{j=0}^{n}(a_k-a_j)=(n+1)(a_i-a_k).\]
                        Hence,  \(n+1=(n+1)\dfrac{a_i-a_k}{m}\), which implies \(m=a_i-a_k.\)
                \end{proof}
                
Now we can prove  Theorem \ref{Thm:Main1}.
                
                \begin{proof}[Proof of Theorem \ref{Thm:Main1}]
                	
                	Let \(\ham\) be a Hamiltonian \(S^1\)-space of dimension \(2n\) which admits a toric \(1\)-skeleton \(\mathcal{S}\) with equivariant pseudo-index \(\rho_M^{S^1}(\mathcal{S})=n+1\). We assume that \(n\leq 5\) or that the vector of even Betti numbers of \(M\) is unimodal.\\
                	By Corollary \ref{Cor:Bettino}, \(M\) has the same Betti numbers as \(\C P^n\). So for \(i=0,\dots,n\) we denote by  \(P_i\)  the unique fixed point with exactly \(i\) negative weights and we denote  by  \(\Gamma_i\)  the sum of all weights at \(P_i\).\\ 
                	By Lemma \ref{Lem:sumGamma}, there exist integers \(a_0>\dots>a_n\), such that
                        \(\Gamma_i=\sum_{j=0}^{n}(a_i-a_j)\). Moreover, if \(a_i-a_j\) is a weight at \(P_i\), then \(a_j-a_i\) is a weight at \(P_j\) by Corollary \ref{Co:possibleweights}.\\
                        Now we prove by induction over the fixed points, that the set of weights at the fixed point \(P_i\) coincides with \(\left\lbrace a_i-a_j \right\rbrace_ {j \neq i}\), i.e. the \(S^1\)-representations at the fixed points are the same as those of a standard \(S^1\)-action on \(\C P^n\). \\
                        Consider the situation at the minimum \(P_0\) of \(\psi\). All weights \(\omega_{0,1}, \dots ,\omega_{0,n}\) at \(P_0\) are positive and  \(a_0-a_j\) is positive for \(j\neq 0\). Hence,
                        \begin{align}\label{Eq:1}
                        \varphi_0(t)=\frac{\prod_{j\neq 0}\left( 1-t^{a_0-a_j}\right)  }
                        {\prod_{k=1}^{n}\left( 1-t^{\omega_{0k}}\right) }
                        \end{align}
                        is a rational function with \(\varphi_0(0)=1\). By Lemma  \ref{Lem:Laurent} \(\varphi_0\) is also a Laurent polynomial. Thus, \(\varphi_0\) is an ordinary polynomial.  Since the numerator and denominator of \(\varphi_0(t)\)  are both polynomials of degree \(\sum_{j=1}^{n}(a_0-a_j)=\Gamma_0=\sum_{k=1}^{n}\omega_{0k}\), we must have that \(\varphi_0\) is constant. In particular, the numerator and denominator  of \(\varphi_0(t)\) have the same roots in \(\C\). This implies
                        \begin{align*}
                        \left\lbrace \omega_{01},\dots, \omega_{0n}\right\rbrace =\left\lbrace a_0-a_j\right\rbrace _{j\neq 0}.
                        \end{align*}\\
                        Now assume that the claim holds for the fixed points \(P_0,\dots,P_{i-1}\). By Corollary  \ref{Co:possibleweights} it follows that the negative weights at \(P_i\) are given by \( a_i-a_0,\, a_i-a_1,\dots, a_i-a_{i-1}\). So let \(\omega_{i1},\dots,\omega_{in}\) be the weights at \(P_i\). Without loss of generality we can assume that \(\omega_{ik}=a_i-a_{k-1}\) for \(k=1,\dots,i\). Hence, \(\omega_{i(i+1)},\dots,\omega_{in}\) are the positive weights at \(P_i\). Moreover, we have
                        \begin{align}\label{summe}
                        \sum_{j=i+1}^{n}(a_i-a_j)=\sum_{j=i+1}^{n}\omega_{ij}
                        \end{align}
                        and
                        \begin{align}
                        \varphi_i(t)=\frac{\prod_{j\neq i}\left( 1-t^{a_i-a_j}\right)  }
                        {\prod_{k=1}^{n}\left(  1-t^{\omega_{ik}}\right)  }=\frac{\prod_{j\geq i+1}\left( 1-t^{a_i-a_j} \right) }
                        {\prod_{k\geq i+1}\left( 1-t^{\omega_{ik}}\right) }.
                        \end{align}
                        Since \(a_i-a_j>0\)  and \(\omega_{ij}>0\) for \(j\geq i+1\), we have \(\varphi_i(0)=1\). Moreover, by Lemma \ref{Lem:Laurent}, \(\varphi_i\) is a Laurent polynomial. So \eqref{summe} and \(\varphi_i(0)=1\) imply that \(\varphi_i\) is constant.
                        Hence, we have
                        \[\left\lbrace a_i-a_j\right\rbrace _{j\geq i+1}=\left\lbrace \omega_{ij}\right\rbrace_{j\geq i+1} \]
                        and the claim follows.
                \end{proof}

        \end{section}

\begin{section}{The Homotopy Type: Proof of Theorem 1.7}
In this section we prove  Theorem \ref{Thm:Main2}. Given a Hamiltonian \(S^1\)-space \(\ham\), which satisfies the conditions of Theorem \ref{Thm:Main2}, Theorem \ref{Thm:Main1} implies that the moment map \(\psi:M\rightarrow\R\) is  a perfect Morse function with exactly one critical point of index \(0,2,\dots2n\). So  Morse Theory (see \cite{Milnor}) implies that \(M\) is homotopy equivalent to a CW-complex with exactly one cell in  dimension \(0,2,\dots,2n\). Moreover, \(M\) has the same integer cohomology ring as \(\C P^n\) by Corollary \ref{Cor:Ring}. \\

 Thus, in order to prove Theorem \ref{Thm:Main2},  we need to show that a CW-complex with exactly one cell in dimension \(0,2,\dots, 2n\) and no cells of other dimensions, which has the same integer cohomology ring as \(\C P^n\), is also homotopy equivalent to \(\C  P^n\). This is the content of Theorem \ref{Thm:HomotopyEquivalence}, which concludes this section.\\

Let \(f:X\rightarrow Y\) be a continuous map. If \(f\) is a homotopy equivalence, then \(f\) induces isomorphisms in all homotopy groups. The converse is true if \(X\) and \(Y\) are CW-complexes. This is known as the  Whitehead Theorem. Combining this with the Hurewicz Theorem yields a useful corollary:
\begin{corollary}\label{Cor:WH}
        Let \(X\) and \(Y\) be simply connected CW-complexes and let \(f: X\rightarrow Y\) be a continuous map, such that
        \begin{align*}
        f_* : H_i(X;\Z) \longrightarrow  H_i(Y;\Z)
        \end{align*}
        is an isomorphism for all \(i\). Then \(f\) is a  homotopy equivalence.
\end{corollary}
A proof of the Whitehead Theorem and of Corollary \ref{Cor:WH} can be found in \cite{hatcher}.\\
Moreover, let \(X\) be a CW-complex and \(Y\) a connected space. If the  \(k\)-th homotopy group \(\pi_k(Y)\) of \(Y\) is trivial, then any continuous map \(X^k\rightarrow Y\) extends to a continuous map \(X^{k+1}\rightarrow Y\), where \(X^k\) resp. \(X^{k+1}\) is the \(k\)- resp. \((k+1)\)-skeleton of \(X\)  (see \cite[Lemma 4.7]{hatcher}).
\begin{theorem}\label{Thm:HomotopyEquivalence}
                Let \(X\) be a CW complex with exactly one cell in dimension \(0,2,4,\dots, 2n\) and no cells of other dimensions. If \(X\) has the same integer cohomology ring as \(\C P^n\), then the spaces are homotopy equivalent.
        \end{theorem}
        \begin{proof}
                Consider the \(2\)-skeleton \(X^2\) of \(X\). It  is a CW-complex with just one \(0\)-cell and one \(2\)-cell, hence it is homeomorphic to the the \(2\)-sphere \(S^2\). Let
                \[\varphi: X^2 \rightarrow \C P^1 \cong S^2\]
                be a homotopy equivalence and
                \[i: \C P^1 \rightarrow \C P^n \]
                the inclusion\footnote{We identify \(\C P^1\) with the subset of \(\C P^n\)  consisting of the points of the form \([z_0,z_1,0,\dots,0]\).}. Since the homotopy groups \(\pi_k(\C P^n)\) are trivial for \(k=3,4,\dots,2n\), the map \(i\circ \varphi\colon X^2 \rightarrow \C P^n\) admits an extension \(f: X \rightarrow \C P^n\). The diagram

                \begin{center}
                        \begin{tikzpicture}[->,>=stealth',shorten >=1pt,auto,node distance=1.0cm,
                        thick,main node/.style={circle,draw,font=\sffamily\Large\bfseries}]
                        \node (E) at (0,0) {$X^2$};
                        \node[right=of E] (F) {$\C P^1$};
                        \node[below=of F] (N) {$\C P^n$};
                        \node[below=of E] (M) {$X$};
                        \draw[->] (E)--(F) node [midway,above] {$\varphi$};
                        \draw[->] (F)--(N) node [midway,right] {$i$};
                        \draw[->] (M)--(N) node [midway,below] {$f$};
                        \draw[->] (E)--(M) node [midway,left] {$j$};
                        \end{tikzpicture}

                \end{center}
                commutes, where \(j: X^2 \rightarrow X \) is the inclusion. We obtain a commutative diagram

                \begin{center}
                        \begin{tikzpicture}[->,>=stealth',shorten >=1pt,auto,node distance=1.0cm,
                        thick,main node/.style={circle,draw,font=\sffamily\Large\bfseries}]
                        \node (E) at (0,0) {$H^2(\C P^n; \Z)$};
                        \node[right=of E] (F) {$H^2(X; \Z)$};
                        \node[below=of F] (N) {$H^2(X^2; \Z)$};
                        \node[below=of E] (M) {$H^2(\C P^1; \Z)$};
                        \draw[->] (E)--(F) node [midway,above] {$f^*$};
                        \draw[->] (F)--(N) node [midway,right] {$j^*$};
                        \draw[->] (M)--(N) node [midway,below] {$\varphi^*$};
                        \draw[->] (E)--(M) node [midway,left] {$i^*$};
                        \end{tikzpicture}

                \end{center}
                in cohomology groups. Since \(X\) and \(\C P^n\) have the structure of a CW-complex with no cells of dimension \(3\), the maps \(i^*\) and \(j^*\) are group isomorphisms in the second cohomology groups. Of course \(\varphi^* \) is a group isomorphism, since \(\varphi\) is a homotopy equivalence. Therefore \(f^*: H^2(\C P^n ;\Z)\rightarrow H^2(X;\Z)\) must be a group isomorphism. But \(H^*(\C P^n; \Z)\) and \(H^*(X; \Z)\) are both isomorphic to \(\Z[x]\slash \left\langle x^{n+1}\right\rangle \) as graded rings, where \(x\) has degree \(2\). Thus,  \(f^*: H^*(\C P^n ;\Z)\rightarrow H^*(X;\Z)\) must be a ring isomorphism. We conclude, that \(f\) induces group isomorphisms  \(f_*: H_k(X ;\Z)\rightarrow H_k(\C P^n;\Z)\) for all \(k\in \N_0\). Hence, \(f\) is a homotopy equivalence.

        \end{proof}

\end{section}

\textsc{Mathematisches Institut, Universität zu Köln, Weyertal 86-90, D-50931, Köln, Germany}\\
E-mail address: \textbf{icharton@math.uni-koeln.de}

\end{document}